\documentclass[11pt,a4paper]{article}
\usepackage{amsfonts,amsgen,amstext,amsbsy,amsopn,amsfonts,amscd}
\usepackage[leqno]{amsmath}
\usepackage[amsmath,amsthm,thmmarks]{ntheorem}
\usepackage{epsf,epsfig}
\usepackage{float}
\usepackage{ebezier,eepic}
\usepackage{color}
\usepackage{tikz}
\usepackage{multirow}
\usepackage{mathrsfs}
\usepackage{graphicx}
\usepackage{subfigure}
\newtheorem{thm}{Theorem}[section]

\newtheorem{prop}[thm]{Proposition}

\newtheorem{lem}[thm]{Lemma}
\newtheorem{example}[thm]{Example}
\newtheorem{false statement}{False statement}
\newtheorem{cor}[thm]{Corollary}
\newtheorem{fact}[thm]{Fact}

\theoremstyle{definition}

\newtheorem{claim}{Claim}

\makeatletter \@addtoreset{equation}{section}

\def\hh{\mathcal{H}}
\def\hm{\mathcal{M}}
\def\hn{\mathcal{N}}
\def\hf{\mathcal{F}}
\def\hg{\mathcal{G}}

\def\hl{\mathcal{L}}
\def\ha{\mathcal{A}}
\def\hb{\mathcal{B}}
\def\hd{\mathcal{D}}
\def\he{\mathcal{E}}
\def\hs{\mathcal{S}}
\def\hht{\mathcal{T}}

\def\hp{\mathcal{P}}

\begin{document}

\title{\bf\Large A Product Version of the Hilton-Milner Theorem}
\date{}
\author{Peter Frankl$^1$, Jian Wang$^2$\\[10pt]
$^{1}$R\'{e}nyi Institute, Budapest, Hungary\\[6pt]
$^{2}$Department of Mathematics\\
Taiyuan University of Technology\\
Taiyuan 030024, P. R. China\\[6pt]
E-mail:  $^1$frankl.peter@renyi.hu, $^2$wangjian01@tyut.edu.cn
}

\maketitle
\begin{abstract}
Two families $\hf,\hg$ of $k$-subsets of $\{1,2,\ldots,n\}$ are called non-trivial cross-intersecting if $F\cap G\neq \emptyset$ for all $F\in \hf, G\in \hg$ and  $\cap \{F\colon F\in \hf\}=\emptyset=\cap \{G\colon G\in\hg\}$. In the present paper, we determine the maximum product of the sizes of two non-trivial cross-intersecting families of   $k$-subsets of $\{1,2,\ldots,n\}$ for $n\geq 4k$, $k\geq 8$, which is a product version of the classical Hilton-Milner Theorem.
\end{abstract}

\section{Introduction}
For a positive integer $n$ let $[n]$ denote the standard $n$-set $\{1,2,\ldots,n\}$. For $1\leq i\leq j\leq n$ set also $[i,j]=\{i,i+1,\ldots,j\}$. For an integer $k$ let $\binom{[n]}{k}$ denote the collection of all subsets of $[n]$. Subsets of $\binom{[n]}{k}$ are called {\it $k$-graphs} or {\it $k$-uniform families}.

A $k$-graph $\hf$ is called {\it $t$-intersecting} if $|F\cap F'|\geq t$ for all $F,F'\in \hf$. Analogously, two $k$-graphs $\hf$ and $\hg$ are called {\it cross $t$-intersecting} if $|F\cap G|\geq t$ for all $F\in \hf$ and $G\in \hg$. In case of $t=1$ we omit the 1 and use the term {\it cross-intersecting}.

One of the central results of extremal set theory is the Erd\H{o}s-Ko-Rado Theorem.

\vspace{6pt}
{\noindent\bf Erd\H{o}s-Ko-Rado Theorem (\cite{EKR}).}   Let $n>k>t>0$ and suppose that $\hf\subset \binom{[n]}{k}$ is $t$-intersecting then for $n\geq n_0(k,t)$,
\begin{align}\label{ineq-ekr}
|\hf| \leq \binom{n-t}{k-t}.
\end{align}

We should mention that the exact value of $n_0(k,t)$ is $(k-t+1)(t+1)$. For $t=1$ it was determined already in \cite{EKR}, for $t\geq 15$ it is due to \cite{F78}. Finally Wilson \cite{W3} closed the gap $2\leq t\leq 14$ with a proof valid for all $t$.

Let us note that the {\it full $t$-star}, $\left\{F\in \binom{[n]}{k}\colon [t]\subset F\right\}$ shows that \eqref{ineq-ekr} is best possible. In general, for a set $T\subset[n]$ let $\hs_T=\left\{S\in \binom{[n]}{k}\colon T\subset S\right\}$ denote the {\it star of $T$}.

For the case $t=1$ the classical Hilton-Milner Theorem is a strong stability result. Recall that a family $\hf$ is called {\it non-trivial} if $\cap \{F\colon F\in \hf\}=\emptyset$.

\vspace{6pt}
{\noindent\bf Hilton-Milner Theorem (\cite{HM67}).} If $n> 2k$ and $\hf\subset \binom{[n]}{k}$ is non-trivial intersecting, then
\begin{align}\label{ineq-nontrival}
|\hf| \leq \binom{n-1}{k-1}- \binom{n-k-1}{k-1} +1 =:h(n,k).
\end{align}
\vspace{6pt}

By now there are many different proofs known for this important result, cf. \cite{Alon, Borg,FFuredi,FT,F2017,GH,Mors} etc.
Let
\[
\hh\hm(n,k) =\left\{F\in \binom{[n]}{k}\colon 1\in F,\ F\cap [2,k+1]\neq \emptyset\right\}\cup \{[2,k+1]\}.
\]
Clearly, $\hh\hm(n,k)$ is a non-trivial intersecting family and it shows that \eqref{ineq-nontrival} is best possible.

Let us state our main result.

\begin{thm}\label{main}
Suppose that $\hf,\hg\subset \binom{[n]}{k}$ are non-trivial cross-intersecting families, $n\geq 4k$, $k\geq 8$. Then
\begin{align}\label{ineq-hfhg2}
|\hf||\hg|\leq  h(n,k)^2=\left(\binom{n-1}{k-1}- \binom{n-k-1}{k-1} +1\right)^2.
\end{align}
\end{thm}

We can prove the same statement for $n\geq 5k$, $k \geq 7$ or $n\geq 8k$, $k\geq 6$ as well  using the same proof.

Besides $\hf=\hg=\hh\hm(n,k)$,  for any  $A,B \in {[2,n]\choose k}$ with $A\cap B\neq \emptyset$ the construction
\begin{align*}
\mathcal{F} = \left\{F\in \binom{[n]}{k}\colon 1\in F,\ F\cap A \neq \emptyset\right\}\cup\{B\},\\[5pt] \mathcal{G} = \left\{G\in \binom{[n]}{k}\colon 1\in G,\ G\cap B \neq \emptyset\right\}\cup \{A\}
\end{align*}
also attains equality in \eqref{ineq-hfhg2}.

In the next section we will state several theorems involving the product of the sizes of cross-intersecting families. Some are important and powerful. However, to the best of our knowledge \eqref{ineq-hfhg2} is the first such result that implies the Hilton-Milner Theorem (just set $\hf=\hg$).

In our proofs, we  need two simple inequalities involving binomial coefficients.

\begin{prop}
Let $n,k,i$ be positive integers. Then
\begin{align}
&\binom{n-i}{k} \geq \left(\frac{n-k-(i-1)}{n-(i-1)}\right)^i \binom{n}{k}, \label{ineq-key}\\[5pt]
&\binom{n-i}{k-i} \binom{n}{k}\leq \binom{n-i+1}{k-i+1} \binom{n-1}{k-1}, \mbox{\rm \ for } n\geq  k\geq i\geq 2.\label{ineq-key2}
\end{align}
\end{prop}

\begin{proof}
Since
\[
\frac{\binom{n-i}{k}}{\binom{n}{k}} = \frac{(n-k)(n-k-1)\ldots(n-k-(i-1))}{n(n-1)\ldots(n-(i-1))}\geq \left(\frac{n-k-(i-1)}{n-(i-1)}\right)^i,
\]
we have \eqref{ineq-key} holds.

For \eqref{ineq-key2}, simply note that
\[
\frac{\binom{n-i}{k-i} \binom{n}{k}}{\binom{n-i+1}{k-i+1} \binom{n-1}{k-1}} =\frac{n(k-i+1)}{k(n-i+1)}=\frac{kn-(i-1)n}{kn-(i-1)k}\leq 1,
\]
and the inequality follows.
\end{proof}

Let us recall the following common notations:
$$\hf(i)=\{F\setminus\{i\}\colon i\in F\in \hf\}, \qquad \hf(\bar{i})= \{F\in\hf: i\notin F\}.$$
Note that $|\hf|=|\hf(i)|+|\hf(\bar{i})|$. For $P\subset Q\subset [n]$, let
\[
\hf(P,Q) = \left\{F\setminus Q\colon F\in\hf,\ F\cap Q=P \right\}.
\]
We also use $\hf(\bar{Q})$ to denote $\hf(\emptyset, Q)$. For $\hf(\{i\},Q)$ we simply write  $\hf(i,Q)$.

\section{Earlier product theorems and some tools}
In this section, we review some earlier  theorems concerning the product of cross-intersecting families. We also recall  Hilton's Lemma and give some corollaries that will be used later.

\begin{thm}[Pyber \cite{Pyber86}]
Suppose that $\hf,\hg\subset \binom{[n]}{k}$ are cross-intersecting, $n\geq 2k$. Then
\begin{align}\label{ineq-pyber}
|\hf||\hg| \leq \binom{n-1}{k-1}^2.
\end{align}
\end{thm}

\begin{thm}[Matsumoto-Tokushige \cite{MT}]
Let $k,\ell$ be positive integers, $n\geq 2k\geq 2\ell$. Suppose that $\hf\subset \binom{[n]}{k}$ and $\hg\subset \binom{[n]}{\ell}$ are cross-intersecting. Then
\begin{align}\label{ineq-mt}
|\hf||\hg| \leq \binom{n-1}{k-1}\binom{n-1}{\ell-1}.
\end{align}
\end{thm}

Note that for the case $n\geq 3k\geq 3\ell$, \eqref{ineq-mt} was already proved by Pyber \cite{Pyber86}. For a short proof of \eqref{ineq-pyber} cf. \cite{FK2017}. In that paper some more precise product results are proven.

\begin{example}
Let $2\leq s\leq k+1$ and define two families
\[
\ha_s =\left\{A\in \binom{[n]}{k}\colon 1\in A, A\cap [2,s]=\emptyset\right\},\ \hb_s=\hs_{\{1\}} \cup \left\{B\in \binom{[n]}{k}\colon [2,s]\subset B\right\}.
\]
\end{example}
It is easy to check that $\ha_s$ and $\hb_s$ are cross-intersecting.

\begin{thm}[\cite{FK2017}]
Let $3\leq s\leq k+1$, $n\geq 2k$. Suppose that $\ha,\hb \subset \binom{[n]}{k}$ are cross-intersecting and \[
|\hb|\geq \binom{n-1}{k-1}+\binom{n-s}{k-s+1}.
\]
Then
\begin{align}\label{ineq-FK17}
|\ha||\hb| \leq \left(\binom{n-1}{k-1}-\binom{n-s}{k-1}\right)\left(\binom{n-1}{k-1}+\binom{n-s}{k-s-1}\right)=|\ha_s||\hb_s|.
\end{align}
\end{thm}

An important tool for proving the above results is the Kruskal-Katona Theorem (\cite{Kruskal,Katona}, cf. \cite{F84} or \cite{Keevash} for short proofs of it).

Daykin \cite{daykin} was the first to show that the Kruskal-Katona Theorem implies the $t=1$ case of the Erd\H{o}s-Ko-Rado Theorem. Hilton \cite{Hilton} gave a very useful reformulation of the Kruskal-Katona Theorem. To state it let us recall the definition of the lexicographic order on $\binom{[n]}{k}$. For two distinct sets $F,G\in \binom{[n]}{k}$ we say that $F$ {\it precedes} $G$ if
\[
\min\{i\colon i\in F\setminus G\}<\min\{i\colon i\in G\setminus F\}.
\]
E.g., $(1,7)$ precedes $(2,3)$. For a positive integer $b$ let $\hl(n,b,m)$ denote the first $m$ members of $\binom{[n]}{b}$.

\vspace{6pt}
{\noindent\bf Hilton's Lemma (\cite{Hilton}).} Let $n,a,b$ be positive integers, $n>a+b$. Suppose that $\ha\subset \binom{[n]}{a}$ and $\hb\subset \binom{[n]}{b}$ are cross-intersecting. Then $\hl(n,a,|\ha|)$ and $\hl(n,b,|\hb|)$ are cross-intersecting as well.
\vspace{6pt}

For a family $\ha \subset\binom{[n]}{a}$ and an integer $b$ define the family of transversals (of size $b$) $\hht^{(b)}(\ha)$  by
\[
\hht^{(b)}(\ha) =\left\{B\in \binom{[n]}{b}\colon B\cap A\neq \emptyset \mbox{ for all } A\in \ha\right\}.
\]
Note that $\ha$ and $\hb$ are cross-intersecting iff $\hb \subset \hht^{(b)}(\ha)$.

Let us use this notation to prove three corollaries of Hilton's Lemma.

\begin{cor}
Let $t\geq 1$.
\begin{align}\label{hiltonLem-1}
\mbox{If } |\ha| \geq \binom{n-1}{a-1}+\ldots+\binom{n-t}{a-1} \mbox{ then } |\hb|\leq \binom{n-t}{k-t}.
\end{align}
\end{cor}

\begin{proof}
Note that
\begin{align*}
\hl(n,a,|\ha|)\supset \hl\left(n,a,\binom{n-1}{a-1}+\ldots+\binom{n-t}{a-1}\right)&=\left\{A\in \binom{[n]}{a}\colon A\cap [t]\neq \emptyset\right\}\\[5pt]
&=:\he(n,a,t).
\end{align*}
Since $\hht^{(b)}(\he(n,a,t))$ is the $t$-star $\{B\in\binom{[n]}{b}\colon [t]\subset B\}$, the statement follows.
\end{proof}

\begin{cor}
\begin{align}\label{hiltonLem-2}
\mbox{If } |\ha| > \binom{n-1}{a-1}-\binom{n-b-1}{a-1} \mbox{ then } |\hb|\leq \binom{n-1}{b-1}.
\end{align}
\end{cor}
\begin{proof}
Suppose indirectly $|\hb|\geq \binom{n-1}{b-1}+1$. Consider
\[
\hl\left(n,b,\binom{n-1}{b-1}+1\right)=\left\{B\in \binom{[n]}{b}\colon 1\in B\right\}\cup \{[2,b+1]\}.
\]
Noting that
\[
\hht^{(a)}\left(\hl\left(n,b,\binom{n-1}{b-1}+1\right)\right) = \left\{A\in \binom{[n]}{a}\colon 1\in A, A\cap [2,b+1]\neq \emptyset\right\}
\]
has size $\binom{n-1}{a-1}-\binom{n-b-1}{a-1}$ the desired contradiction follows.
\end{proof}

\begin{cor}
\begin{align}\label{hiltonLem-3.1}
\mbox{If } |\ha| > \binom{n-1}{a-1}+\binom{n-2}{a-1}+\binom{n-4}{a-2} \mbox{ then } |\hb|\leq \binom{n-3}{b-3}+\binom{n-4}{k-3}.
\end{align}
\end{cor}
\begin{proof}
Note that
\begin{align*}
\hl(n,a,|\ha|)\supset &\hl\left(n,a,\binom{n-1}{a-1}+\binom{n-2}{a-1}+\binom{n-4}{a-2}\right)\\[5pt]
=&\left\{A\in \binom{[n]}{a}\colon 1\in   A \mbox{ or }2\in   A\mbox{ or }\{3,4\}\subset  A \right\}=:\hd(n,a).
\end{align*}
Then clearly
\[
\hht^{(b)}(\hd(n,a)) =\left\{B\in\binom{[n]}{b}\colon \{1,2,3\}\subset B \mbox{ or } \{1,2,4\}\subset B\right\}
\]
and the statement follows.
\end{proof}

For later use let us prove one more consequence of Hilton's Lemma. We state it in the special case that we need in Section 3.

\begin{lem}
Suppose that $\hf, \hg\subset \binom{[n]}{k}$ are cross-intersecting, $n>2k$. Let
\[
\binom{n-4}{k-4}<|\hf|\leq \binom{n-3}{k-3},
\]
and
\[
\binom{n-1}{k-1}+ \binom{n-2}{k-1}+\binom{n-3}{k-1}<|\hg|\leq \binom{n-1}{k-1}+ \binom{n-2}{k-1}+\binom{n-3}{k-1}+\binom{n-4}{k-1}.
\]
Define
\[
f=|\hf|-\binom{n-4}{k-4}\mbox{ and } g= |\hg| -\binom{n-1}{k-1}- \binom{n-2}{k-1}-\binom{n-3}{k-1}.
\]
 Then
\begin{align}\label{hiltonLem-3}
\frac{g}{\binom{n-4}{k-1}}+\frac{f}{\binom{n-4}{k-3}}\leq 1.
\end{align}
\end{lem}

\begin{proof}
Note that all
\[
F\in \hl(n,k,|\hf|)\setminus  \hl\left(n,k,\binom{n-4}{k-4}\right)
\]
satisfy $F\cap [4]=[3]$. Let $\hf_0$ be the family of size $f$ formed by the corresponding sets $F\setminus [4]\in \binom{[5,n]}{k-3}$.

Also, for all
\[
G\in \hl(n,k,|\hg|)\setminus  \hl\left(n,k,\binom{n-1}{k-1}+\binom{n-2}{k-1}+\binom{n-3}{k-1}\right),
\]
$G\cap [4]=\{4\}$ follows from $g\leq \binom{n-4}{k-1}$ (a consequence of \eqref{hiltonLem-1}). Let $\hg_0\subset \binom{[5,n]}{k-1}$ be formed by the corresponding sets $G\setminus [4]$. Obviously, $\hf_0$ and $\hg_0$ are cross-intersecting. The inequality \eqref{hiltonLem-3} is essentially due to Sperner \cite{Sperner} but let us repeat the simple argument. Define a bipartite graph with partite sets $\binom{[5,n]}{k-3}$ and $\binom{[5,n]}{k-1}$ by  putting an edge  between $F\in\binom{[5,n]}{k-3}$ and $G\in \binom{[5,n]}{k-1}$ iff $F\cap G=\emptyset$. This bipartite graph is bi-regular implying that the neighborhood $\hn(\hf_0)$ satisfies
\[
|\hn(\hf_0)|/\binom{n-4}{k-1}\geq |\hf_0|/\binom{n-4}{k-3}.
\]
Since $\hf_0\cup \hg_0$ is an independent set, $\hn(\hf_0)\cap \hg_0 =\emptyset$ implying
\[
\frac{|\hg_0|}{\binom{n-4}{k-1}}+\frac{|\hf_0|}{\binom{n-4}{k-3}}\leq 1.
\]
\end{proof}

\section{Some size restrictions for the general case and the proof for shifted-resistant families}

Throughout the proof we assume that $\hf,\hg\subset \binom{[n]}{k}$ are non-trivial, cross-intersecting and
\begin{align}\label{indirectAssum}
|\hf||\hg| \geq h(n,k)^2.
\end{align}

\begin{prop}
For $n\geq 2k-1$, $k\geq 3$
\begin{align}\label{ineq-f1}
\left(\binom{n-3}{k-3}+\binom{n-4}{k-3}\right)&
\left(\binom{n-1}{k-1}+\binom{n-2}{k-1}+\binom{n-3}{k-1}\right)\nonumber\\[5pt]
&\qquad<\left(\binom{n-2}{k-2}+\binom{n-3}{k-2}+\binom{n-4}{k-2}\right)^2.
\end{align}
\end{prop}

\begin{proof}
Note that
\[
\binom{n-1}{k-1}+\binom{n-2}{k-1}+\binom{n-3}{k-1} \leq \frac{n-1}{k-1}\left(\binom{n-2}{k-2}+\binom{n-3}{k-2}+\binom{n-4}{k-2}\right).
\]
Since $n\geq 2k-1$ implies $\frac{n-1}{k-1}\leq \frac{n-3}{k-2}<\frac{n-2}{k-2}$, we infer
\begin{align*}
\frac{n-1}{k-1}\left(\binom{n-3}{k-3}+\binom{n-4}{k-3}\right)&\leq \frac{n-2}{k-2}\binom{n-3}{k-3}+\frac{n-3}{k-2}\binom{n-4}{k-3}\leq \binom{n-2}{k-2}+ \binom{n-3}{k-2}.
\end{align*}
Thus \eqref{ineq-f1} follows.
\end{proof}

\begin{prop}
In proving Theorem \ref{main} we may assume that
\begin{align}
&\min \left\{|\hf|,|\hg|\right\}> \binom{n-3}{k-3}+\binom{n-4}{k-3},\label{ineq-1}\\[5pt]
&\max \left\{|\hf|,|\hg|\right\}\leq \binom{n-1}{k-1}+\binom{n-2}{k-1}+\binom{n-4}{k-2}.\label{ineq-2}
\end{align}
\end{prop}
\begin{proof}
By symmetry assume that $|\hf|\leq |\hg|$. If $|\hf|< \binom{n-4}{k-4}$, then applying \eqref{ineq-key2} twice we obtain
\[
|\hf||\hg|<\binom{n-4}{k-4}\binom{n}{k} <\binom{n-2}{k-2}^2 <h(n,k)^2.
\]
If $|\hg|> \binom{n}{k}-\binom{n-4}{k}$
 then by \eqref{hiltonLem-2} we have $|\hf|< \binom{n-4}{k-4}$. Thus we may assume that $|\hf|\geq \binom{n-4}{k-4}$ and $|\hg|\leq \binom{n}{k}-\binom{n-4}{k}$.

Let
\[
|\hf|=\binom{n-4}{k-4}+\alpha\binom{n-4}{k-3},\ |\hg| =\binom{n-1}{k-1}+\binom{n-2}{k-1}+\binom{n-3}{k-1}+\beta\binom{n-4}{k-1}.
\]
Note that  \eqref{hiltonLem-3} implies  $\alpha+\beta\leq 1$ and thereby
\[
|\hf| \leq \binom{n-3}{k-3}-\beta \binom{n-4}{k-3}.
\]
Since $\frac{n-k}{n-3}>\frac{n-k-2}{n-3}$ implies
\[
\frac{\binom{n-4}{k-3}}{\binom{n-3}{k-3}} >\frac{\binom{n-4}{k-1}}{\binom{n-3}{k-1}}>\frac{\binom{n-4}{k-1}}{\binom{n-1}{k-1}+\binom{n-2}{k-1}+\binom{n-3}{k-1}},
\]
it follows that
\begin{align}
|\hf||\hg| &\leq\left(\binom{n-3}{k-3}-\beta \binom{n-4}{k-3}\right)\left(\binom{n-1}{k-1}+\binom{n-2}{k-1}+\binom{n-3}{k-1}+\beta\binom{n-4}{k-1}\right)\nonumber\\[5pt]
&\leq \binom{n-3}{k-3}\left(\binom{n-1}{k-1}+\binom{n-2}{k-1}+\binom{n-3}{k-1}\right).\label{ineq-hfhgab}
\end{align}
By \eqref{ineq-key2} we know
\[
\binom{n-3}{k-3}\binom{n-1}{k-1}<\binom{n-2}{k-2}^2.
\]
Moreover,
\begin{align*}
\binom{n-3}{k-3}\binom{n-2}{k-1} &\leq \binom{n-2}{k-2}\binom{n-3}{k-2} \frac{k-2}{n-2}\frac{n-2}{k-1} < \binom{n-2}{k-2}\binom{n-3}{k-2},\\[5pt]
\binom{n-3}{k-3}\binom{n-3}{k-1} &\leq \binom{n-2}{k-2}\binom{n-4}{k-2} \frac{k-2}{n-2}\frac{n-3}{k-1} < \binom{n-2}{k-2}\binom{n-4}{k-2}.
\end{align*}
Thus from \eqref{ineq-hfhgab} and $k\geq 3$ we obtain
\begin{align*}
|\hf||\hg| \leq \binom{n-2}{k-2}\left(\binom{n-2}{k-2}+\binom{n-3}{k-2}+\binom{n-4}{k-2}\right)<h(n,k)^2,
\end{align*}
contradicting \eqref{indirectAssum}. Thus  we may assume that
\[
|\hf|\geq \binom{n-3}{k-3} \mbox{ and } |\hg|\leq \binom{n-1}{k-1}+\binom{n-2}{k-1}+\binom{n-3}{k-1}.
\]

If $|\hf|\leq \binom{n-3}{k-3}+\binom{k-4}{k-3}$, then by \eqref{ineq-f1}
\[
|\hf||\hg|\leq \left(\binom{n-2}{k-2}+\binom{n-3}{k-2}+\binom{n-4}{k-2}\right)^2< h(n,k)^2,
\]
contradicting \eqref{indirectAssum} again. Thus we may further assume $|\hf| > \binom{n-3}{k-3}+\binom{n-4}{k-3}$. Then by \eqref{hiltonLem-3.1} it implies that
\[
|\hg|\leq \binom{n-1}{k-1}+\binom{n-2}{k-1}+\binom{n-4}{k-2}.
\]
\end{proof}

We need the following computational bound to estimate the size of $\hf$ and $\hg$ below.
\begin{lem}
For $n\geq 4k$ and $k\geq 8$,
\begin{align}
h(n,k)&>\frac{32}{9}\binom{n-2}{k-2}>3\binom{n-2}{k-2},\label{ineq-hmn-2k-22}\\[5pt]
h(n,k)&> \frac{9}{4}\binom{n-2}{k-2}+\frac{9}{4}\binom{n-4}{k-2}.\label{ineq-hmn-2k-2}
\end{align}
\end{lem}
\begin{proof}
Since $n\geq (j-1)k/2$ implies $\frac{n-k-j+3}{n-j+1}\geq \frac{n-k}{n}$, by \eqref{ineq-key} and $n\geq 4k$ we infer
\[
\frac{\binom{n-j}{k-2}}{\binom{n-2}{k-2}}\geq \left(\frac{n-k-j+3}{n-j+1}\right)^{j-2} \geq  \left(\frac{n-k}{n}\right)^{j-2}\geq \left(\frac{3}{4}\right)^{j-2},\  2\leq j\leq 9.
\]
It follows that
\begin{align*}
\sum_{j=2}^9 \binom{n-j}{k-2}&\geq  \binom{n-2}{k-2}\sum_{j=2}^9 \left(\frac{3}{4}\right)^{j-2}=\frac{1-\left(\frac{3}{4}\right)^8}{1-\frac{3}{4}}
\binom{n-2}{k-2}> \frac{32}{9}\binom{n-2}{k-2}.
\end{align*}
Thus by $k\geq 8$ we have
\[
h(n,k)>\binom{n-1}{k-1}-\binom{n-k-1}{k-1}\geq \sum_{j=2}^9 \binom{n-j}{k-2}> \frac{32}{9}\binom{n-2}{k-2}.
\]

By \eqref{ineq-key} and $n\geq 4k$ we also have
\begin{align}\label{ineq-3}
\sum_{j=2}^4 \binom{n-j}{k-2}&\geq  \binom{n-2}{k-2} \left(1+\frac{3}{4}+\left(\frac{3}{4}\right)^2\right)>\frac{9}{4}\binom{n-2}{k-2}.
\end{align}
Moreover,
\[
\frac{\binom{n-j}{k-2}}{\binom{n-4}{k-2}}\geq \left(\frac{n-k-j+3}{n-j+1}\right)^{j-4} \geq  \left(\frac{n-k}{n}\right)^{j-4}\geq \left(\frac{3}{4}\right)^{j-4},\ 5\leq j\leq 9.
\]
We infer
\begin{align}\label{ineq-4}
\sum_{j=5}^9 \binom{n-j}{k-2}&\geq \binom{n-4}{k-2}\left(\frac{3}{4}+\left(\frac{3}{4}\right)^2+\left(\frac{3}{4}\right)^3
+\left(\frac{3}{4}\right)^4+\left(\frac{3}{4}\right)^5\right)> \frac{9}{4}\binom{n-4}{k-2}.
\end{align}
Adding \eqref{ineq-3} and \eqref{ineq-4} we obtain \eqref{ineq-hmn-2k-2}.
\end{proof}

To prove the theorem we apply shifting, a powerful method that can be traced back to \cite{EKR}. For $\hf\subset{[n]\choose k}$ and $1\leq i< j\leq n$, define the shift
$$S_{ij}(\hf)=\left\{S_{ij}(F)\colon F\in\hf\right\},$$
where
$$S_{ij}(F)=\left\{
                \begin{array}{ll}
                  (F\setminus\{j\})\cup\{i\}, & j\in F, i\notin F \text{ and } (F\setminus\{j\})\cup\{i\}\notin \hf; \\[5pt]
                  F, & \hbox{otherwise.}
                \end{array}
              \right.
$$
It is well known (cf. \cite{F87}) that shifting preserves the  cross-intersecting property.

Let us define the {\it shifting partial order} $\prec$. For two $k$-sets $A$ and $B$ where $A=\{a_1,\ldots,a_k\}$, $a_1<\ldots<a_k$ and $B=\{b_1,\ldots,b_k\}$, $b_1<\ldots<b_k$ we say that $A$ precedes $B$ and denote it by $A\prec B$ if $a_i\leq b_i$ for all $1\leq i\leq k$.

A family $\hf\subset \binom{[n]}{k}$ is called {\it shifted} (or {\it initial}) if $A\prec B$ and $B\in \hf$ always imply $A\in \hf$. By repeated shifting one can transform an arbitrary $k$-graph into a shifted $k$-graph with the same number of edges.

The only problem with shifting is that it might destroy the non-triviality of $\hf$ or $\hg$ or both.

\begin{fact}\label{fact-3.1}
If $S_{ij}(\hf)\subset \hs_{\{i\}}$,
then
\begin{itemize}
  \item[(i)] $\hf(\overline{ij})=\emptyset$ and
  \item[(ii)] $\hf(i)\cap \hf(j)=\emptyset$. (Note that this is equivalent with $\hf(\overline{i}j)\cap \hf(i\overline{j})=\emptyset$).
\end{itemize}
\end{fact}

%Noting that $|\hf|=|S_{ij}(\hf)|$ and $|\hs_{\{i\}}|\leq \binom{n-1}{k-1}$ we infer
%
%\begin{obs}
%If $S_{ij}(\hf)\subset \hs_{\{i\}}$ then
%\[
%|\hf| \leq \binom{n-1}{k-1}.
%\]
%\end{obs}
%
%Similarly we have
%\begin{prop}\label{prop-2.3}
%Let $Q\subset [n]$ and  $(i,j)\subset[n]\setminus Q$. Suppose that  $S_{ij}(\hf)\subset \hs_{\{i\}}$. Then for all $P\subset Q$,
%\begin{align}\label{ineq-hfpq}
%|\hf(P,Q)|\leq \binom{n-|Q|-1}{k-|P|-1}.
%\end{align}
%\end{prop}

%Let us prove one more simple inequality.
%
%\begin{lem}
%For every  $x\in [n]$,
%\begin{align}\label{ineq-degree}
%|\hg(x)| <h(n,k).
%\end{align}
%\end{lem}
%\begin{proof}
%Simply note that $\hf\not\subset\hs_{\{x\}}$ implies the existence of $F\in \hf$ with $x\notin F$. Now \eqref{ineq-degree} follows from $\left(G\setminus \{x\}\right) \cap F\neq \emptyset$ for all $G\in \hg$.
%\end{proof}

With all this preparation we are ready to state and prove the main result of this section.

\begin{prop}\label{lem-2.4}
Let $n\geq  4k$, $k\geq 8$ and $\hf,\hg\subset \binom{[n]}{k}$ be non-trivial cross-intersecting. Suppose that there exist disjoint pairs $(a,b), (c,d)$ such that $S_{ab}(\hf)\subset \hs_{a}$ and  $S_{cd}(\hf)\subset \hs_{c}$. Then
\begin{align}\label{ineq-main}
|\hf||\hg|< h(n,k)^2.
\end{align}
\end{prop}

\begin{proof}
For notational convenience assume that $(a,b)=(1,2)$ and $(c,d)=(3,4)$.
 Arguing indirectly we assume \eqref{indirectAssum}. For $i=2,3,4$, define
\[
\hf_i = \{F\setminus [4]\colon F\in \hf,\ |F\cap [4]|=i\}.
\]
Since $S_{ij}(\hf)\subset \hs_{\{i\}}$ for $(i,j)=(1,2)$ and $(3,4)$ implies $(1,2),(3,4)\in \hht^{(2)}(\hf)$, we have
\begin{align}\label{ineq-fcap5n}
|F\cap [5,n]|\leq k-2 \mbox{ for all }F \in \hf.
\end{align}

\begin{claim}\label{claim-1}
For every $E\in \hf_2$, there are at most two choices for $S\in \binom{[4]}{2}$ with $E\cup S\in \hf$. For every $T\in \hf_3$, there are at most two choices for $R\in \binom{[4]}{3}$ with $T\cup R\in \hf$.
\end{claim}
\begin{proof}
Suppose that for some $E\in \hf_2$ there  are three choices for $S\in \binom{[4]}{2}$ with $E\cup S\in \hf$.  Note that Fact \ref{fact-3.1} (i) implies $S\neq (1,2)$ and $S\neq (3,4)$. Choose $S_1,S_2$ such that $S_1\cap S_2\neq \emptyset$. Without loss of generality, assume that $S_1=(1,3)$ and $S_2=(1,4)$, then $E\cup \{1\}  \in \hf(3)\cap \hf(4)$, contradicting Fact \ref{fact-3.1} (ii).

Similarly, suppose that for some  $T\in \hf_3$ there are three choices for $R\in \binom{[4]}{3}$ with $T\cup R\in \hf$. Then we may choose $R_1,R_2$ such that $R_1\cap R_2= (1,2)$ or $(3,4)$. Without loss of generality, assume that $R_1=(1,2,3)$ and $R_2=(1,2,4)$, then $T\cup \{1,2\}  \in \hf(3)\cap \hf(4)$, contradicting Fact \ref{fact-3.1} (ii).
\end{proof}

By Claim \ref{claim-1} we have
\begin{align}\label{ineq-hfhp}
|\hf| \leq 2|\hf_2|+2|\hf_3|+|\hf_4|.
\end{align}
It follows that
\begin{align}\label{ineq-hfupbound}
|\hf| &\leq  2\binom{n-4}{k-2}+2\binom{n-4}{k-3}+\binom{n-4}{k-4}=\binom{n-2}{k-2}+\binom{n-4}{k-2}.
\end{align}
By \eqref{ineq-hmn-2k-2}, for $n\geq 4k$, $k\geq 8$ we have $|\hf| < \frac{4}{9}h(n,k)$. Then the indirect assumption \eqref{indirectAssum} implies
\begin{align}\label{ineq-2.5hnk}
|\hg| >\frac{9}{4}h(n,k).
\end{align}

Let
\[
\hp=\left\{P\in \binom{[4]}{2}\colon P\cap (1,2)\neq \emptyset, P\cap (3,4)\neq \emptyset\right\}.
\]

\begin{claim}\label{claim-new3.0}
\begin{align}\label{ineq-new3.0}
\sum_{P\in \hp}|\hf(P,[4])|> 4\binom{n-5}{k-3}.
\end{align}
\end{claim}

\begin{proof}
Suppose for contradiction that $\sum\limits_{P\in \hp}|\hf(P,[4])|\leq 4\binom{n-5}{k-3}$. Note that Fact \ref{fact-3.1} (i) implies $|\hf(\{1,2\},[4])|=|\hf(\{3,4\},[4])|=0$.  Then by Claim \ref{claim-1} we have
\begin{align*}
|\hf|&\leq \sum_{P\in \hp}|\hf(P,[4])|+2|\hf_3|+|\hf_4|\\[5pt]
 &\leq 4\binom{n-5}{k-3}+2\binom{n-4}{k-3}+ \binom{n-4}{k-4}\\[5pt]
 &=4\binom{n-5}{k-3}+\binom{n-4}{k-3}+\binom{n-3}{k-3}\\[5pt]
 &< 3\binom{n-3}{k-3}+3\binom{n-5}{k-3}.
\end{align*}
By \eqref{ineq-hmn-2k-2} it follows that
\[
|\hf| \leq 3\left(\frac{k-2}{n-2}\binom{n-2}{k-2}+\frac{k-2}{n-4}\binom{n-4}{k-2}\right)< \frac{4(k-1)}{3(n-1)} h(n,k).
\]
By \eqref{indirectAssum} and \eqref{ineq-hmn-2k-22},
\begin{align}\label{ineq-sumtotal}
|\hg| >\frac{3(n-1)}{4(k-1)} h(n,k)>\frac{3(n-1)}{4(k-1)} 3\binom{n-2}{k-2}>2\binom{n-1}{k-1},
\end{align}
contradicting \eqref{ineq-2}.
\end{proof}

\begin{claim}
\begin{align}\label{ineq-hg4barub}
|\hg(\overline{[4]})|\leq \binom{n-7}{k-3}+\binom{n-8}{k-3}.
\end{align}
\end{claim}
\begin{proof}
By Claim \ref{claim-1} and \eqref{ineq-new3.0}, we infer
\[
2|\hf_2| \geq \sum_{P\in \hp}|\hf(P,[4])|> 4\binom{n-5}{k-3}.
\]
It follows that $|\hf_2|>2\binom{n-5}{k-3}> \binom{n-5}{k-3}+\binom{n-6}{k-3}+\binom{n-8}{k-4}$. Since $\hf_2$, $\hg(\overline{[4]})$ are cross-intersecting, by \eqref{hiltonLem-3.1} we obtain \eqref{ineq-hg4barub}.
\end{proof}

\begin{claim}
For $i\in [4]$,
\begin{align}\label{ineq-newhgi}
|\hg(i,[4])|\leq \binom{n-4}{k-1} - \binom{n-2-k}{k-1}.
\end{align}
\end{claim}

\begin{proof}
Since $\hf$ is non-trivial, there exists $F\in \hf$ such that $i\notin F$. Let $E=F\cap [5,n]$. By \eqref{ineq-fcap5n}, $|E|\leq k-2$. Now the cross-intersection implies $E\cap E'\neq \emptyset$ for each $E' \in \hg(i,[4])$. Thus the claim follows.
\end{proof}

\begin{claim}\label{claim-new3}
At most five of $\hg(P,[4])$, $P\in \binom{[4]}{2}$ have size greater than $\binom{n-5}{k-3}-\binom{n-3-k}{k-3}$.
\end{claim}

\begin{proof}
Suppose for contradiction that $|\hg(P,[4])|>\binom{n-5}{k-3}-\binom{n-3-k}{k-3}$ for all $P\in \binom{[4]}{2}$. Then for each $P\in \binom{[4]}{2}$ let $P'=[4]\setminus P$. Apply \eqref{hiltonLem-2} with $a=b=k-2$ to the cross-intersecting families $\hg(P,[4])$ and $\hf(P',[4])$, we infer $|\hf(P',[4])|\leq \binom{n-5}{k-3}$.
Then
\[
\sum_{P\in \hp}|\hf(P,[4])|\leq 4\binom{n-5}{k-3},
\]
contradicting \eqref{ineq-new3.0}.

\end{proof}

By Claim \ref{claim-new3}, we infer that
\begin{align}\label{ineq-claim2p2}
\sum_{P\subset[4], |P|\geq 2} |\hg(P,[4])|&\leq 5\binom{n-4}{k-2}+\binom{n-5}{k-3}-\binom{n-3-k}{k-3}+4\binom{n-4}{k-3}+\binom{n-4}{k-4}\nonumber\\[5pt]
&<2 \binom{n-4}{k-2}+3\left(\binom{n-4}{k-2}+\binom{n-4}{k-3}\right)+\left(\binom{n-4}{k-3} +\binom{n-4}{k-4}\right)\nonumber\\[5pt]
&\qquad+\binom{n-5}{k-3}\nonumber\\[5pt]
&=2 \binom{n-4}{k-2}+3\binom{n-3}{k-2}+\binom{n-3}{k-3}+\binom{n-5}{k-3}\nonumber\\[5pt]
&=2 \binom{n-4}{k-2}+2\binom{n-3}{k-2}+\binom{n-2}{k-2}+\binom{n-5}{k-3}\nonumber\\[5pt]
&<2\binom{n-2}{k-2}+2\binom{n-3}{k-2} +\binom{n-4}{k-2}.
\end{align}

\begin{claim}\label{claim-2}
Exactly two of $\hg(1,[4])$, $\hg(2,[4])$, $\hg(3,[4])$, $\hg(4,[4])$ have size greater than $\binom{n-5}{k-2}-\binom{n-3-k}{k-2}$.
\end{claim}

\begin{proof}
Suppose that at least three of $\hg(1,[4])$, $\hg(2,[4])$, $\hg(3,[4])$, $\hg(4,[4])$ have size greater than $\binom{n-5}{k-3}-\binom{n-3-k}{k-3}$. Then for each $P\in \hp$ there exists $i\in [4]$ such that $i\notin P$ and $|\hg(i,[4])|> \binom{n-5}{k-2}-\binom{n-3-k}{k-2}$. Apply \eqref{hiltonLem-2} to $\hg(i,[4])$ and $\hf(P,[4])$ with $a=k-1$ and $b=k-2$, we infer $|\hf(P,[4])|\leq \binom{n-5}{k-3}$ for all $P\in \hp$, contradicting \eqref{ineq-new3.0}. Thus, at most two of $\hg(1,[4])$, $\hg(2,[4])$, $\hg(3,[4])$, $\hg(4,[4])$ have size greater than $\binom{n-5}{k-3}-\binom{n-3-k}{k-3}$.

Suppose that at most one of them has size greater than $\binom{n-5}{k-2}-\binom{n-3-k}{k-2}$. Then by \eqref{ineq-newhgi} we have
\begin{align*}
\sum_{ 1\leq  i\leq 4} |\hg(i,[4])|\leq \binom{n-4}{k-1} - \binom{n-2-k}{k-1}+3\left(\binom{n-5}{k-2}-\binom{n-3-k}{k-2}\right).
\end{align*}
Using
\[
\binom{n-5}{k-2}-\binom{n-3-k}{k-2}\leq\binom{n-4}{k-2}-\binom{n-2-k}{k-2},
\]
we get
\begin{align}\label{ineq-claim2p1}
\sum_{ 1\leq  i\leq 4} |\hg(i,[4])|&\leq \binom{n-4}{k-1} - \binom{n-2-k}{k-1}+\binom{n-4}{k-2} - \binom{n-2-k}{k-2}\nonumber\\[5pt]
&\qquad+2\left(\binom{n-5}{k-2}-\binom{n-3-k}{k-2}\right)\nonumber\\[5pt]
&< \binom{n-3}{k-1} - \binom{n-1-k}{k-1}+2\binom{n-5}{k-2}\nonumber\\[5pt]
&< h(n,k)- \binom{n-2}{k-2}- \binom{n-3}{k-2} +2\binom{n-5}{k-2}.
\end{align}
Adding \eqref{ineq-hg4barub}, \eqref{ineq-claim2p2} and \eqref{ineq-claim2p1},
\begin{align*}
|\hg|=&\sum_{\emptyset\neq P\subset[4]} |\hg(P,[4])|+|\hg(\overline{[4]})|\\[5pt]
<&h(n,k)+\binom{n-2}{k-2}+\binom{n-3}{k-2}+\binom{n-4}{k-2}+2\binom{n-5}{k-2}
+\binom{n-7}{k-3}+\binom{n-8}{k-3}.
\end{align*}
Note that $k\geq 7$ implies $h(n,k)\geq \sum\limits_{2\leq i\leq 8} \binom{n-i}{k-2}$. It follows that
\begin{align*}
|\hg|<&2h(n,k)+\binom{n-5}{k-2}-\binom{n-6}{k-2}-\binom{n-7}{k-2}-\binom{n-8}{k-2}
+\binom{n-7}{k-3}+\binom{n-8}{k-3}\\[5pt]
=&2h(n,k)+\binom{n-6}{k-3}+\binom{n-7}{k-3}+\binom{n-8}{k-3}
-\binom{n-7}{k-2}-\binom{n-8}{k-2}\\[5pt]
\leq &2h(n,k)+2\binom{n-6}{k-3}-2\binom{n-8}{k-2}.
\end{align*}
Since $n\geq 4k$  and $k\geq 3$ imply
\begin{align*}
\frac{\binom{n-8}{k-2}}{\binom{n-6}{k-3}} &=\frac{(n-k-3)(n-k-4)(n-k-5)}{(n-6)(n-7)(k-2)}\\[5pt]
&\geq \frac{(3k-3)(3k-4)(3k-5)}{(4k-6)(4k-7)(k-2)}>\left(\frac{3}{4}\right)^2\times3>1.
\end{align*}
we infer $|\hg|<2h(n,k)$, contradicting \eqref{ineq-2.5hnk}.
\end{proof}

Now consider a pair $(i,P)$, $P\in \binom{[4]}{2}$, $i\in [4]\setminus P$ and $|\hg(i,[4])|> \binom{n-5}{k-2}-\binom{n-3-k}{k-2}$. Apply \eqref{hiltonLem-2} to $\hg(i,[4])$ and $\hf(P,[4])$ with $a=k-1$, $b=k-2$, we infer $|\hf(P,[4])|\leq \binom{n-5}{k-3}$. By Claim \ref{claim-2}, exactly two of $\hg(1,[4])$, $\hg(2,[4])$, $\hg(3,[4])$, $\hg(4,[4])$ have size greater than $\binom{n-5}{k-2}-\binom{n-3-k}{k-2}$. Hence, at most one  $P\in \{(1,3),(1,4),(2,3),(2,4)\}$ satisfies $|\hf(P,[4])|> \binom{n-5}{k-3}$. It follows that
\[
\sum_{P\in \hp} |\hf(P,[4])| \leq \binom{n-4}{k-2}+3\binom{n-5}{k-3}.
\]
Then by Claim \ref{claim-1}  and the identity $\binom{n-2}{k-2}=\binom{n-4}{k-2}+2\binom{n-4}{k-3}+\binom{n-4}{k-4}$,
\begin{align*}
|\hf| &\leq \sum_{P\in \hp}|\hf(P,[4])|+2|\hf_3|+|\hf_4|\\[5pt]
&\leq 3\binom{n-5}{k-3}+\binom{n-4}{k-2}+2\binom{n-4}{k-3}+ \binom{n-4}{k-4} \\[5pt]
&=\binom{n-2}{k-2}+ 3\binom{n-5}{k-3}.
\end{align*}
Since  for $n\geq 4k$,
\begin{align*}
|\hf|\leq  \binom{n-2}{k-2}+ 3\binom{n-5}{k-3}=\binom{n-2}{k-2}+\frac{3(k-2)}{n-4}\binom{n-4}{k-2}<\binom{n-2}{k-2}+\frac{3}{4}\binom{n-4}{k-2},
\end{align*}
 by \eqref{ineq-hmn-2k-22} and \eqref{ineq-hmn-2k-2} we infer
\begin{align}\label{ineq-hf1}
|\hf| &< \frac{1}{4}\binom{n-2}{k-2}+ \frac{3}{4} \left(\binom{n-2}{k-2} + \binom{n-4}{k-2}\right)\nonumber\\[5pt]
&<\frac{1}{4}\times \frac{9}{32}h(n,k)+\frac{3}{4}\times \frac{4}{9}h(n,k)\nonumber\\[5pt]
& = \frac{155}{384}h(n,k)<\frac{32}{73}h(n,k).
\end{align}

By Claim \ref{claim-2} and \eqref{ineq-newhgi},
\begin{align}\label{ineq-claim3p1}
\sum_{1\leq i\leq 4}|\hg(i,[4])|
&\leq 2\left(\binom{n-4}{k-1} - \binom{n-2-k}{k-2}\right)+2\left(\binom{n-5}{k-2}-\binom{n-3-k}{k-2}\right)\nonumber\\[5pt]
&\leq 2\left(\binom{n-3}{k-1} - \binom{n-1-k}{k-2}\right)\nonumber\\[5pt]
&= 2h(n,k)-2\binom{n-2}{k-2}-2\binom{n-3}{k-2}.
\end{align}
Adding  \eqref{ineq-hg4barub}, \eqref{ineq-claim2p2} and \eqref{ineq-claim3p1},
\begin{align}\label{ineq-hg1}
|\hg|&=|\hg(\overline{[4]})|+\sum_{\emptyset\neq P\subset[4]} |\hg(P,[4])|\nonumber\\[5pt]
& \leq 2h(n,k)+\binom{n-4}{k-2}+\binom{n-7}{k-3}+\binom{n-8}{k-3}\nonumber\\[5pt]
&\leq 2h(n,k)+\binom{n-2}{k-2}\overset{\eqref{ineq-hmn-2k-22}}< \frac{73}{32}h(n,k).
\end{align}
But now  \eqref{ineq-hf1} and \eqref{ineq-hg1} imply $|\hf||\hg|<h(n,k)^2$, contradicting \eqref{indirectAssum}. This concludes the proof of the proposition.
\end{proof}

\section{The shifted case and the proof of the main theorem}

In this section, we determine the maximum product of sizes of non-trivial shifted cross-intersecting families.

First, we determine the maximum sum of sizes of non-trivial shifted cross-intersecting families $\hf,\hg$ with  $(2,5,7,\ldots,2k-1,2k+1)\notin \hf\cup \hg$  by modifying an injective map introduced in \cite{F2017}.

\begin{prop}
Let $n\geq k>0$ and $\hf,\hg\subset \binom{[n]}{k}$ be cross-intersecting. Suppose that both $\hf$ and $\hg$ are non-trivial shifted families. Moreover $(2,5,7,\ldots,2k-1,2k+1)\notin \hf\cup \hg$. Then
\begin{align}\label{ineq-hfhg2hm}
|\hf|+|\hg|\leq 2h(n,k).
\end{align}
\end{prop}

\begin{proof}
We claim that for every $H\in \hf\cup \hg$ there exists $\ell$ such that $|H\cap [2\ell]|\geq \ell$. Moreover, if $1\notin H$ then $\ell \geq 2$. Arguing indirectly, assume that no such  $\ell$  exists for $H=\{a_1,a_2,\ldots,a_k\}$
where $a_1<a_2<\ldots<a_k$. For $\ell=1$ this implies $a_1\geq 2$. For $\ell\geq 2$ we infer $a_\ell> 2\ell$ for all $\ell=2,\ldots,k$. By shiftedness, we see that
$(2,5,7,\ldots,2k+1)\in \hf\cup \hg$, a contradiction.  Let $\ell(H)$ be the maximal $\ell$ such that $|H\cap [2\ell]|\geq \ell$.

Recall the notation $A\triangle B = (A\setminus B)\cup (B\setminus A)$, the symmetric difference. Define the function $\phi\colon \phi(H) = H\triangle [2\ell(H)]$.

\begin{claim}\label{claim-4}
For $G,G'\in \hg$, $\phi(G)\neq \phi(G')$.
\end{claim}

\begin{proof}
If $\ell(G)=\ell(G'):=\ell$ then $\phi(G)\triangle [2\ell]=G\neq G'=\phi(G') \triangle [2\ell]$ implies $\phi(G)\neq \phi(G')$. On the other hand if $\ell(G)>\ell(G')$ then
\[
|\phi(G) \cap [2\ell(G)]|=\ell >|\phi(G') \cap [2\ell(G)]|
\]
by the maximality of $\ell(G')$.
\end{proof}

\begin{claim}\label{claim-5}
For $G\in \hg(\bar{1})$, $\phi(G)\setminus \{1\}\notin \hf(1)$.
\end{claim}

\begin{proof}
Set $\ell=\ell(G)$. Let $G=(x_1,x_2,\ldots,x_k)$ with $x_1<x_2<\ldots<x_k$. The maximal choice of $\ell$ implies
\[
[2\ell]\cap G =\{x_1,\ldots,x_\ell\}, \ x_{\ell+1}>2\ell+2,\ldots,x_k>2k.
\]
By shiftedness $(x_1,\ldots,x_\ell)\cup (2\ell+2,2\ell+4,\ldots,2k)\in \hg$. Note that $G\in \hg(\bar{1})$ implies $x_1\geq 2$. If $\phi(G)\setminus \{1\}=([2,2\ell]\setminus (x_1,\ldots,x_\ell))\cup (2\ell+2,2\ell+4,\ldots,2k)\in \hf(1)$ then by shiftedness  $([2\ell]\triangle (x_1,\ldots,x_\ell))\cup (2\ell+1,2\ell+3,\ldots,2k-1)\in \hf$ and it contradicts the cross-intersecting property.
\end{proof}

\begin{claim}\label{claim-6}
If $H\neq [2,k+1]$, $H\in \hf\cup \hg$, $1\notin H$ then $\phi(H)\cap [2,k+1]\neq \emptyset$.
\end{claim}
\begin{proof}
There are two cases. Let $\ell=\ell(H)$. If $2\ell\geq k+1$ then $[2,k+1]\subset [2\ell]$. Consequently, we can choose $x\in [2,k+1]\setminus H$ since $H\neq [2,k+1]$. Thus $x\in [2\ell]\setminus H\subset H\triangle [2\ell]$, i.e., $x\in [2,k+1]\cap \phi(H)$. As we noted before, $1\notin H$ and $(2,5,7,\ldots,2k-1,2k+1)\notin \hf\cup \hg$ imply that $\ell\geq 2$. If $k+1>2\ell$ then $\ell\geq 2$ implies the existence of $x\in  [2,2\ell]\setminus H$ and hence $x\in H\triangle [2\ell]$. It follows that $x\in [2,k+1]\cap \phi(H)$.
\end{proof}

Note that the non-triviality and the shiftedness imply $[2,k+1]\in \hf\cap \hg$. From Claims \ref{claim-4}, \ref{claim-5} and \ref{claim-6}, we infer
\begin{align}
|\hf(1)|+|\hg(\bar{1})| &=|\hf(1)|+|\phi\left(\hg(\bar{1})\setminus \{[2,k+1]\}\right)|+1\nonumber\\[5pt]
&\leq \binom{n-1}{k-1}-\binom{n-k-1}{k-1}+1.\label{ineq-hf1hgbar}
\end{align}
Switching the roles of $\hf$ and $\hg$, we obtain
\begin{align}
&|\hf(\bar{1})|+|\hg(1)| \leq \binom{n-1}{k-1}-\binom{n-k-1}{k-1}+1.\label{ineq-hfbarhg1}
\end{align}
Adding \eqref{ineq-hf1hgbar} and \eqref{ineq-hfbarhg1}, we get \eqref{ineq-hfhg2hm}.
\end{proof}

Recall the following inequality from \cite{F78}, for a proof using linear algebra cf. \cite{F2022}.

%By a nice combination of the shiftedness property and the lattice path counting argument, the first author proved the following useful inequality in \cite{F78}. We also refer to \cite{F2022} for a simple linear algebraic proof.

\begin{lem}[\cite{F78}]\label{lem-walk}
Let $\hf\subset \binom{[n]}{k}$ be a shifted family with $0\leq t<k$. If $[t]\cup \{t+2,t+4,\ldots,2k-t\}\notin \hf$, then
\begin{align}\label{ineq-walk}
|\hf| \leq \binom{n}{k-t-1}.
\end{align}
\end{lem}

\begin{prop}\label{lem-2.6}
Let $\hf,\hg\subset \binom{[n]}{k}$ be non-trivial and cross-intersecting, $n\geq 4k$, $k\geq 8$. If both $\hf$ and $\hg$ are shifted, then
\begin{align}\label{ineq-shiftedhfhg2}
|\hf||\hg|\leq h(n,k)^2.
\end{align}
\end{prop}

\begin{proof}
We may assume that $(1,3,5,\ldots,2k-1)\notin \hf\cap \hg$. Indeed, if $(1,3,5,\ldots,2k-1)\in \hf\cap \hg$, then by cross-intersection we infer $(2,4,\ldots,2k)\notin \hf\cup\hg$. By shiftedness, it follows that
 $(2,5,7,\ldots,2k-1,2k+1)\notin \hf\cup \hg$. By  \eqref{ineq-hfhg2hm},
 \[
 |\hf||\hg|\leq \left(\frac{|\hf|+|\hg|}{2}\right)^2\leq h(n,k)^2
 \]
 and we are done.

By symmetry assume that $(1,3,5,\ldots,2k-1)\notin \hf$. Then by \eqref{ineq-walk}  we have
\begin{align}\label{ineq-f781}
|\hf| \leq \binom{n}{k-2}.
\end{align}
By \eqref{ineq-1} and $n\geq  4k$, we obtain that
\[
|\hf|>\binom{n-3}{k-3}+\binom{n-4}{k-3}=\frac{n-3}{k-3}\binom{n-4}{k-4}+\frac{n-k}{k-3} \binom{n-4}{k-4} >7\binom{n-4}{k-4}.
\]
By  \eqref{ineq-key} and $n\geq 4k$,
\[
|\hf|> 7\binom{n-4}{k-4}>\left(\frac{4}{3}\right)^4\binom{n-4}{k-4}> \left(\frac{n-3}{n-k+1}\right)^4\binom{n-4}{k-4}> \binom{n}{k-4}.
\]
Then Lemma  \ref{lem-walk} implies that $(1,2,3,5,7,\ldots,2k-3)\in \hf$.

We distinguish two cases.

\vspace{6pt}
{\noindent \bf Case 1.} $(1,2,4,6,\ldots,2k-2)\notin \hf$.
\vspace{6pt}

In view of \eqref{ineq-walk} the assumption $(1,2,4,6,\ldots,2k-2)\notin \hf$ implies for $n\geq 4k$
\begin{align}\label{ineq-f782}
|\hf| \leq \binom{n}{k-3}\overset{\eqref{ineq-key}}{<}\left(\frac{n-2}{n-k+1}\right)^3\binom{n-3}{k-3}
<\left(\frac{4}{3}\right)^3\binom{n-3}{k-3}<3\binom{n-3}{k-3}.
\end{align}
By \eqref{ineq-2},
\begin{align}\label{ineq-f783}
|\hg| \leq \binom{n-1}{k-1}+\binom{n-2}{k-1}+\binom{n-4}{k-2}< 2\binom{n-1}{k-1}.
\end{align}

From \eqref{ineq-f782} and \eqref{ineq-f783}, we obtain
\begin{align}\label{ineq-hfhg3}
|\hf||\hg|<6\binom{n-1}{k-1}\binom{n-3}{k-3}\overset{\eqref{ineq-key2}}{\leq} 6\binom{n-2}{k-2}^2\overset{\eqref{ineq-hmn-2k-22}}<h(n,k)^2.
\end{align}

\vspace{6pt}
{\noindent \bf Case 2.}  $(1,2,4,6,\ldots,2k-2)\in \hf$.
\vspace{6pt}

By cross-intersection, $(3,5,\ldots,2k+1)\notin \hg$. Using that $\hf$ is non-trivial, $[2,k+1]\in \hf$ follows. By cross-intersection again $G\cap [2,k+1]\neq \emptyset$ for all $G\in \hg$. Consequently
\[
|\hg(1)|\leq \binom{n-1}{k-1}-\binom{n-k-1}{k-1}<h(n,k).
\]
Consider $\hg(\bar{1},2)\subset \binom{[3,n]}{k-1}$. Since $\hf(\bar{2})\neq \emptyset$ and $\hf$ is shifted, $\hf(1,\bar{2})\neq \emptyset$. Thus,
\[
|\hg(\bar{1},2)|\leq \binom{n-2}{k-1} - \binom{n-k-1}{k-1}.
\]

For $G\in \hg(\bar{1},\bar{2})$ define $\tilde{G}=\{x-2\colon x\in G\}$ and set
\[
\tilde{\hg}=\{\tilde{G}\colon G\in \hg(\bar{1},\bar{2})\}.
\]
Note that $(3,5,\ldots,2k+1)\notin \hg$ implies $(1,3,\ldots,2k-1)\notin \tilde{\hg}$. Applying \eqref{ineq-walk} with $t=1$ yields
\begin{align}\label{ineq-g1bar2bar}
|\hg(\bar{1},\bar{2})|\leq \binom{n-2}{k-2}.
\end{align}Thus,
\begin{align}\label{ineq-hg}
|\hg| \leq |\hg(1)|+|\hg(\bar{1},2)|+|\hg(\bar{1},\bar{2})|<2h(n,k).
\end{align}
By \eqref{ineq-f781} and \eqref{ineq-key}, we obtain for $n\geq 4k$
\begin{align}\label{ineqhfnk-2}
|\hf|\leq \binom{n}{k-2}\leq \left(\frac{n-1}{n-k+1}\right)^2\binom{n-2}{k-2}< \left(\frac{4}{3}\right)^2\binom{n-2}{k-2}=\frac{16}{9}\binom{n-2}{k-2}.
\end{align}
Combining \eqref{ineqhfnk-2} and \eqref{ineq-hg}, we obtain
\begin{align*}
|\hf||\hg|<\frac{32}{9}\binom{n-2}{k-2}h(n,k) \overset{\eqref{ineq-hmn-2k-22}}{<}   h(n,k)^2.
\end{align*}
\end{proof}

In order to prove Theorem \ref{main} let us introduce the notion of shift-resistant pair. For a pair of families $\hf,\hg\subset \binom{[n]}{k}$, define the quantity
\[
w(\hf,\hg) =\sum_{F\in \hf}\sum_{i\in F} i + \sum_{G\in \hg}\sum_{j\in G} j.
\]
Let us fix  $\hf,\hg\subset \binom{[n]}{k}$ so that $\hf,\hg$ are non-trivial cross-intersecting, $|\hf||\hg|$ is maximal, moreover among such pairs $w(\hf,\hg)$ is minimal. If in such pair  $\hf$ and $\hg$ are not both shifted then we say that $(\hf,\hg)$ is shift-resistant.

\begin{prop}\label{prop-4.4}
Suppose that $(\hf,\hg)$ is a shift-resistant cross-intersecting pair. Then for $\hh=\hf$ or $\hh=\hg$ there exist disjoint pairs $(a,b)$, $(c,d)$ such that $S_{ab}(\hh)\subset \hs_a$ and $S_{cd}(\hh)\subset \hs_c$.
\end{prop}
\begin{proof}
Since $\hf$ and $\hg$ are not both shifted, without loss of generality suppose $S_{ab}(\hf)\subset \hs_a$. Note that this implies $F\cap (a,b)\neq \emptyset$ for all $F\in \hf$. Hence
\[
\hs_{\{a,b\}}=\left\{S\in \binom{[n]}{k}\colon (a,b)\subset S\right\}\subset \hg
\]
follows from the maximality of $|\hf||\hg|$.

Consider an arbitrary pair $(c,d) \subset [n]\setminus (a,b)$. First note that $S_{cd}(\hg)\subset \hs_c$ cannot hold. Indeed, $\hs_{\{a,b\}}\subset \hg$ implies $\hg(\bar{c},\bar{d})\neq \emptyset$ and thereby $S_{cd}(\hg)\not\subset \hs_c$. If $S_{cd}(\hf)\subset \hs_c$ then we are done. Thus we may assume that $S_{cd}(\hg)\not\subset \hs_c$. Now the minimality of $w(\hf,\hg)$ implies $S_{cd}(\hf)=\hf$ and $S_{cd}(\hg)=\hg$ for all $(c,d)\subset [n]\setminus (a,b)$.

Let $D=(d_1,d_2,\ldots,d_{k-1})$ be the $k-1$ smallest elements of $[n]\setminus (a,b)$. Then $\hf(\bar{a},\bar{b})=\emptyset$ and $\hf(\bar{a})\neq \emptyset$ imply $\hf(\bar{a},b) \neq \emptyset$, moreover $S_{cd}(\hf)=\hf$ for all $(c,d) \subset [n]\setminus (a,b)$ implies $D\in \hf(\bar{a},b)$. Similarly, $D\in \hf(a,\bar{b})$. Thus $\hf(\bar{a},b)\cap \hf(a,\bar{b})\neq \emptyset$, contradicting $S_{ab}(\hf)\subset \hs_a$.
\end{proof}

\begin{proof}[Proof of Theorem \ref{main}]
Let $\hf,\hg\subset \binom{[n]}{k}$ be  a pair of non-trivial cross-intersecting families with $|\hf||\hg|$ maximal. Moreover, we assume that $w(\hf,\hg)$ is minimal among all such pairs. Then the minimality of $w(\hf,\hg)$ implies that either both of $\hf$, $\hg$ are shifted or $(\hf,\hg)$ forms a shift-resistant pair.If $\hf$ and $\hg$ are both shifted then by Proposition \ref{lem-2.6}  we conclude that $|\hf||\hg|\leq  h(n,k)^2$. 

If $(\hf,\hg)$ forms a shift-resistant pair, then by Proposition \ref{prop-4.4} there exist  disjoint pairs $(a,b)$, $(c,d)$ such that  either $\hf$ or $\hg$ (call it $\hh$) satisfies  $S_{ab}(\hh)\subset \hs_a$ and $S_{cd}(\hh)\subset \hs_c$. By symmetry assume that $S_{ab}(\hf)\subset \hs_a$ and $S_{cd}(\hf)\subset \hs_c$. Then by Proposition \ref{lem-2.4} we conclude that $|\hf||\hg|< h(n,k)^2$.
\end{proof}

\section{ Further improvements and concluding remarks}

The most intriguing problem is whether \eqref{ineq-hfhg2} of Theorem \ref{main} holds for the full range, that is, for all $n,k$ satisfying $n\geq 2k\geq 4$.  Note that for $n=2k$,
 $|\hf|+|\hg|\leq {2k \choose k}=2h(n,k)$ is obvious. Consequently, $|\hf| |\hg| \leq  h(n,k)^2$.

%We made quite an effort to reduce our original bounds from $n> 20k$ to the present $n\geq 4k$.   To attack the case $2k<n< 4k$, no doubt, one needs some new ideas.

By a different argument, we can also prove Theorem \ref{main} for $n\geq 3k$ and $k\geq 15$ as well. Let us first prove a statement of some independent interest.

\begin{prop}\label{prop-5.1}
Suppose that $\hf,\hg\subset \binom{[n]}{k}$ are cross-intersecting,  $n\geq 2k$ and
$\min\{|\hf|,|\hg|\}\geq \binom{n-3}{k-3}+\binom{n-4}{k-3}$. Then
\begin{align}
|\hf|+|\hg|\leq 2\binom{n-1}{k-1}.
\end{align}
\end{prop}

\begin{proof}
By Hilton's Lemma, without loss of generality, assume that $\hf=\hl(n,k,|\hf|),\hg=\hl(n,k,|\hg|)$, $|\hf|\leq |\hg|$ and $|\hg|>\binom{n-1}{k-1}$. Thus $|\hg(1)|=\binom{n-1}{k-1}$ and $|\hf(1)|=|\hf|$. If $|\hg(\bar{1})|\leq\binom{n-2}{k-1}$ then
\[
|\hf|+|\hg|= |\hf(1,2)|+|\hf(1,\bar{2})|+|\hg(1)|+|\hg(\bar{1},2)|.
\]
Since $\hf(1,\bar{2}),\hg(\bar{1},2)\subset\binom{[3,n]}{k-1}$ are cross-intersecting, we have
\[
|\hf(1,\bar{2})|+|\hg(\bar{1},2)| \leq \binom{n-2}{k-1}.
\]
It follows that
\begin{align*}
|\hf|+|\hg|&= |\hf(1,2)|+|\hg(1)|+(|\hf(1,\bar{2})|+|\hg(\bar{1},2)|)\\[5pt]
&\leq \binom{n-2}{k-2}+\binom{n-1}{k-1}+\binom{n-2}{k-1}\\[5pt]
&=2\binom{n-1}{k-1}.
\end{align*}
Thus we may assume that $|\hg|>\binom{n-1}{k-1}+\binom{n-2}{k-1}$.

Now
\[
|\hg|=\binom{n-1}{k-1}+\binom{n-2}{k-1}+|\hg(\bar{1},\bar{2})|\mbox{ and }|\hf|=|\hf(1,2)|.
\]
By the assumption on $|\hf|$, $\hf(1,2)\subset\binom{[n-2]}{k-2}$ satisfies $|\hf(1,2)|\geq \binom{n-3}{k-3}+\binom{n-4}{k-3}$. Consequently $\{3,4\}\subset G$ for all  $G\in \hg(\bar{1},\bar{2})$. We infer $|\hg(\bar{1},\bar{2},3,4)|=|\hg(\bar{1},\bar{2})|$ and
\[
|\hf(1,2)|= \binom{n-3}{k-3}+\binom{n-4}{k-3}+|\hf(1,2,\bar{3},\bar{4})|.
\]
As $\hf(1,2,\bar{3},\bar{4})$ and $\hg(\bar{1},\bar{2},3,4)$  are cross-intersecting, we have
\[
|\hf(1,2,\bar{3},\bar{4})|+|\hg(\bar{1},\bar{2},3,4)|\leq \binom{n-4}{k-2}.
\]
Consequently,
\begin{align*}
|\hf|+|\hg|&= \binom{n-3}{k-3}+\binom{n-4}{k-3}+|\hf(1,2,\bar{3},\bar{4})|+\binom{n-1}{k-1}
+\binom{n-2}{k-1}+|\hg(\bar{1},\bar{2},3,4)|\\[5pt]
 &\leq \binom{n-1}{k-1}+ \binom{n-2}{k-1}+ \binom{n-3}{k-3}+\binom{n-4}{k-3}+\binom{n-4}{k-2}\\[5pt]
 &=2\binom{n-1}{k-1}.
\end{align*}
\end{proof}

\begin{lem}\label{lem-5.2}
Let $n\geq  2k$ and let $\hf,\hg\subset \binom{[n]}{k}$ be non-trivial cross-intersecting. Let $|\hf|=\alpha h(n,k)$ and suppose that $\alpha <1$. Set $f(\alpha) =\frac{1+\alpha^2}{(1-\alpha)^2}$. Then $|\hf||\hg|> h(n,k)^2$ implies
\begin{align}\label{ineq-keyfalpha}
f(\alpha) \geq \prod_{1\leq i\leq k-1} \frac{n-i}{n-k-i}.
\end{align}
\end{lem}
\begin{proof}
 By \eqref{ineq-1} and Proposition \ref{prop-5.1}, we infer $|\hf|+|\hg|\leq 2\binom{n-1}{k-1}$. Note that $|\hf||\hg|> h(n,k)^2$ implies  $|\hg|\geq h(n,k)/\alpha$. Therefore,
\begin{align*}
2\binom{n-1}{k-1} \geq |\hf|+|\hg| &\geq \left(\alpha+\frac{1}{\alpha}\right)h(n,k)\geq \left(\alpha+\frac{1}{\alpha}\right)\left(\binom{n-1}{k-1}-\binom{n-k-1}{k-1}\right).
\end{align*}
By rearranging,
\[
\left(\alpha+\frac{1}{\alpha}-2\right)\binom{n-1}{k-1}\leq \left(\alpha+\frac{1}{\alpha}\right)\binom{n-k-1}{k-1}.
\]
Multiplying both sides by $\alpha$ we get
\[
(1-\alpha)^2\binom{n-1}{k-1}\leq (1+\alpha^2)\binom{n-k-1}{k-1}.
\]
Thus,
\begin{align*}
f(\alpha)=\frac{1+\alpha^2}{(1-\alpha)^2} \geq \frac{\binom{n-1}{k-1}}{\binom{n-k-1}{k-1}}
=\frac{(n-1)(n-2)\ldots(n-k+1)}{(n-k-1)(n-k-2)\ldots(n-2k+1)}.
\end{align*}
\end{proof}

In case of shift-resistant families and $k\geq 9$ we succeeded to extend the proof to the whole range.

\begin{prop}\label{prop-5.3.0}
Let $n\geq  2k+1$, $k\geq 6$ and $\hf,\hg\subset \binom{[n]}{k}$ be non-trivial cross-intersecting. Suppose that  $(\hf,\hg)$ forms a shift-resistant pair. Set $f(\alpha) =\frac{1+\alpha^2}{(1-\alpha)^2}$. If
\[
\prod\limits_{1\leq i\leq k-1} \frac{n-i}{n-k-i} > f(0.64)\approx 10.8765,
\]
then
\begin{align*}
|\hf||\hg|\leq   h(n,k)^2.
\end{align*}
\end{prop}

\begin{proof}
By Proposition \ref{prop-4.4} and \eqref{ineq-hfupbound}, we have
\begin{align}
|\hf| \leq  \binom{n-2}{k-2}+\binom{n-4}{k-2}.
\end{align}
For $n=2k+1$,
\begin{align*}
\frac{\binom{n-1}{k}+\binom{n-4}{k}}{\binom{n}{k}} &=\frac{k+1}{2k+1}\left(\frac{2k(2k-1)2(k-1)+k(k-1)(k-2)}{2k(2k-1)2(k-1)}\right)\\[5pt]
&=\frac{k+1}{2k+1}\frac{3(3k-2)}{4(2k-1)}  =\frac{3}{4} \frac{3k^2+k-2}{4k^2-1} >\frac{9}{16}.
\end{align*}
Since $\binom{x-d}{k}/\binom{x}{k}$ is an increasing function of $x$, for all $n\geq 2k+1$,
\begin{align}\label{ineq-key3}
\binom{n}{k}+\binom{n-1}{k}+\binom{n-4}{k}>\frac{25}{16}\binom{n}{k}.
\end{align}
By \eqref{ineq-key3} we infer for $(n-2)\geq 2(k-2)+1$,
\[
\binom{n-2}{k-2}+\binom{n-3}{k-2}+\binom{n-6}{k-2}>\frac{25}{16}\binom{n-2}{k-2}.
\]
For $(n-4)\geq 2(k-2)+1$,
\[
\binom{n-4}{k-2}+\binom{n-5}{k-2}+\binom{n-7}{k-2}>\frac{25}{16}\binom{n-4}{k-2}.
\]
Therefore,
\[
h(n,k) \geq \sum_{2\leq i\leq 7} \binom{n-i}{k-2}\geq \frac{25}{16}\left(\binom{n-2}{k-2}+\binom{n-4}{k-2}\right)\geq \frac{25}{16} |\hf|,
\]
implying that $\alpha =\frac{|\hf|}{h(n,k)} \leq \frac{16}{25}=0.64$.  Since $f(\alpha)$ is increasing in $[0,1]$,  we obtain
\[
f(\alpha)\leq f\left(0.64\right)< \prod\limits_{1\leq i\leq k-1} \frac{n-i}{n-k-i}.
\]
Then Lemma \ref{lem-5.2} implies $|\hf||\hg|\leq   h(n,k)^2$.
\end{proof}

\begin{prop}
For $n\geq 2k$,
\begin{align}\label{ineq-key4}
\prod\limits_{1\leq i\leq k-1} \frac{n-i}{n-k-i} >\left(\frac{n-\frac{k}{2}}{n-\frac{3k}{2}}\right)^{k-1}.
\end{align}
\end{prop}

\begin{proof}
Note that for $m>d>i>0$,
\begin{align}\label{ineq-key5}
\frac{m-d-i}{m-i}\cdot \frac{m-d+i}{m+i} <\left(\frac{m-d}{m}\right)^2.
\end{align}
Equivalently,
\begin{align*}
&\frac{(m-d)^2-i^2}{(m-d)^2}<\frac{m^2-i^2}{m^2}, \mbox{ that is} \\[5pt]
&\left(\frac{i}{m}\right)^2<\left(\frac{i}{m-d}\right)^2,
\end{align*}
which is true for $m>d>0$.
Applying \eqref{ineq-key5} repeatedly with $m=n-\frac{k}{2}$ and $d=k$, we obtain
\[
\frac{(n-k-1)(n-k-2)\ldots (n-2k+1)}{(n-1)(n-2)\ldots (n-k+1)} <\left(\frac{n-\frac{3k}{2}}{n-\frac{k}{2}}\right)^{k-1}
\]
and \eqref{ineq-key4} follows.
\end{proof}

\begin{cor}\label{prop-5.2}
Let $n\geq  2k+1$, $k\geq 9$ and $\hf,\hg\subset \binom{[n]}{k}$ be non-trivial cross-intersecting. If $(\hf,\hg)$ forms a shift-resistant pair, then $|\hf||\hg|\leq   h(n,k)^2$.
\end{cor}

\begin{proof}
By Theorem \ref{main} we may assume that $2k\leq n\leq 4k$. By \eqref{ineq-key4} and $k\geq 9$, it follows that
\[
\prod\limits_{1\leq i\leq k-1} \frac{n-i}{n-k-i} >\left(\frac{n-\frac{k}{2}}{n-\frac{3k}{2}}\right)^{k-1}\geq \left(\frac{4k-\frac{k}{2}}{4k-\frac{3k}{2}}\right)^8=\left(\frac{7}{5}\right)^8\approx 14.7579>f(0.64).
\]
By Proposition \ref{prop-5.3.0}, we conclude that $|\hf||\hg|\leq   h(n,k)^2$.
\end{proof}

\begin{cor}\label{prop-5.2.2}
Let $2k+1 \leq n\leq  3.13k$, $k\geq 6$ and $\hf,\hg\subset \binom{[n]}{k}$ be  non-trivial cross-intersecting. If $(\hf, \hg)$ forms a shift-resistant pair, then $|\hf||\hg|\leq   h(n,k)^2$.
\end{cor}
\begin{proof}
By \eqref{ineq-key4} and $k\geq 6$, it follows that
\[
\prod\limits_{1\leq i\leq k-1} \frac{n-i}{n-k-i} >\left(\frac{n-\frac{k}{2}}{n-\frac{3k}{2}}\right)^{k-1}\geq \left(\frac{3.13k-\frac{k}{2}}{3.13k-\frac{3k}{2}}\right)^5=\left(\frac{2.63}{1.63}\right)^5\approx 10.9356>f(0.64).
\]
By Proposition \ref{prop-5.3.0}, we conclude that $|\hf||\hg|\leq   h(n,k)^2$.
\end{proof}

Similarly, we can prove the same statement for  $2k+1 \leq n\leq  3.54k$, $k=7$ or $2k+1 \leq n\leq  3.96k$, $k=8$.

\begin{prop}\label{prop-5.4}
Let $n\geq  3k$, $k\geq 14$ and $\hf,\hg\subset \binom{[n]}{k}$ be non-trivial cross-intersecting. If both $\hf$ and $\hg$ are  shifted, then
\begin{align*}
|\hf||\hg|\leq  h(n,k)^2.
\end{align*}
\end{prop}

\begin{proof}
By Theorem \ref{main} we may assume that $3k\leq n\leq 4k$. Suppose that $|\hf|\leq |\hg|$. By \eqref{ineq-f781} we have $|\hf|\leq \binom{n}{k-2}$. Set $\alpha=\frac{|\hf|}{h(n,k)}$. Then
\[
\frac{1}{\alpha}=\frac{h(n,k)}{|\hf|} \geq \frac{\binom{n-1}{k-1}-\binom{n-k-1}{k-1}}{\binom{n}{k-2}}.
\]
Note that $3k\leq n\leq 4k$ and $k\geq 14$ imply
\[
\frac{\binom{n-1}{k-1}}{\binom{n}{k-2}} = \frac{(n-k+2)(n-k+1)}{(k-1)n} \geq \frac{4}{3}
\]
and
\begin{align*}
\frac{\binom{n-k-1}{k-1}}{\binom{n}{k-2}}
&=\frac{(n-k-1)(n-k-2)\ldots(n-2k+1)}{(k-1)n\ldots(n-k+3)}\\[5pt]
&\leq \frac{n-k-1}{k-1} \left(\frac{n-k-2}{n}\right)^{k-2}\\[5pt]
&\leq \left(3+\frac{2}{k-1}\right)\left(\frac{3}{4}\right)^{k-2}\\[5pt]
&< 3.16\times\left(\frac{3}{4}\right)^{12}<\frac{1}{9}.
\end{align*}
It follows that $\frac{1}{\alpha} \geq \frac{4}{3}-\frac{1}{9}=\frac{11}{9}$, implying  $\alpha \leq \frac{9}{11}$. Thus,
\[
f(\alpha) \leq f\left(\frac{9}{11}\right)=50.5.
\]
By \eqref{ineq-key4} and $k\geq 14$, it follows that
\[
\prod\limits_{1\leq i\leq k-1} \frac{n-i}{n-k-i} >\left(\frac{n-\frac{k}{2}}{n-\frac{3k}{2}}\right)^{k-1}\geq \left(\frac{4k-\frac{k}{2}}{4k-\frac{3k}{2}}\right)^{13}=\left(\frac{7}{5}\right)^{13}\approx 79.3917>f\left(\frac{9}{11}\right).
\]
By Lemma \ref{lem-5.2} the proposition follows.
\end{proof}

Applying Corollary \ref{prop-5.2},  Proposition \ref{prop-5.4} and repeating the  proof of Theorem \ref{main}, we obtain the following result.

\begin{thm}\label{main3}
Suppose that $\hf,\hg\subset \binom{[n]}{k}$ are non-trivial cross-intersecting families, $n\geq 3k$, $k\geq 14$. Then $|\hf||\hg|\leq  h(n,k)^2$.
\end{thm}

The next proposition rules out many potential constructions that may prevent Theorem \ref{main} holding for the full range.

\begin{prop}\label{prop-5.3}
Let $n\geq 2k\geq 4$. Suppose that $\hf,\hg\subset\binom{[n]}{k}$, non-trivial, cross-intersecting. If $\hht^{(2)}(\hf)\cap\hht^{(2)}(\hg)\neq \emptyset$ then
\begin{align}\label{ineq-prop-0}
|\hf|+|\hg|\leq 2h(n,k).
\end{align}
\end{prop}
\begin{proof}
Without loss of generality, assume $(1,2)\cap H\neq \emptyset$ for all $H\in \hf\cup \hg$. Obviously,
\begin{align}\label{ineq-prop-3}
|\hf([2])|+|\hg([2])|\leq 2\binom{n-2}{k-2}.
\end{align}
Note that $\hf(1,\bar{2})$, $\hg(1,\bar{2})$, $\hf(\bar{1},2)$ and $\hg(\bar{1},2)$ are non-empty by non-triviality. $\hf(1,\bar{2})$ and $\hg(\bar{1},2)$ are cross-intersecting $(k-1)$-graphs on $[3,n]$. Consequently, by a classical result of Hilton and Milner \cite{HM67} concerning non-empty cross-intersecting families we obtain
\begin{align}\label{ineq-prop-1}
|\hf(1,\bar{2})|+|\hg(\bar{1},2)|\leq\binom{n-2}{k-1}-\binom{n-k-1}{k-1}+1.
\end{align}
The same is true for  $\hf(\bar{1},2)$ and $\hg(1,\bar{2})$:
\begin{align}\label{ineq-prop-2}
|\hf(\bar{1},2)|+|\hg(1,\bar{2})|\leq \binom{n-2}{k-1}-\binom{n-k-1}{k-1}+1.
\end{align}
Summing up \eqref{ineq-prop-3}, \eqref{ineq-prop-1} and \eqref{ineq-prop-2} yields
\[
|\hf|+|\hg|\leq 2\left(\binom{n-2}{k-2}+\binom{n-2}{k-1}-\binom{n-k-1}{k-1}+1\right)=2h(n,k).
\]
\end{proof}

Let us mention that the case $k=2$, that is, if $\hf$ and $\hg$ are  non-trivial cross-intersecting graphs then $|\hf||\hg|\leq 9$ is easy to prove. If $\hf$ and $\hg$ share an edge then $|\hf|+|\hg|\leq 6$ follows from Proposition \ref{prop-5.1} and thereby $|\hf||\hg|\leq 9$. Arguing indirectly assume $|\hf|\geq 4$ and let $(x,y)$ be an edge of $\hg$. Non-triviality of $\hg$ implies that both $x$ and $y$ have degree two in $\hf$. Let $u$ and $v$ be the neighbors of $x$ in $\hf$. Then the only candidate for an edge in $\hg$ that does not contain $x$ is $(u,v)$ whence $(u,v)\in\hg$. By cross-intersection the two neighbors of $y$ in $\hf$ must be $u$ and $v$ as well. Now cross-intersection implies $\hg=\{(x,y),(u,v)\}$ whence $|\hf||\hg|=8<9$.

{\bf\noindent Remark.} There are many cases of equality in \eqref{ineq-prop-0}. Namely, let $A,B\in \binom{[3,n]}{k-1}$. Set
\[
\hf(1,\bar{2}) = \{ A\},\ \hg(\bar{1},2)=\left\{G\in \binom{[3,n]}{k-1}\colon G\cap A\neq\emptyset\right\}.
\]
There are two possibilities for $B$.
\begin{itemize}
  \item[(a)] $\hf(\bar{1},2) = \{B\},\ \hg(1,\bar{2})=\left\{G\in \binom{[n]}{k-1}\colon G\cap B\neq\emptyset\right\}$,
  \item[(b)] $\hf(\bar{1},2) = \left\{F\in \binom{[n]}{k-1}\colon F\cap B\neq\emptyset\right\},\ \hg(1,\bar{2})=\{B\}$.
\end{itemize}
Add $\hf(1,2)=\hg(1,2)=\binom{[3,n]}{k-2}$. Clearly $|\hf||\hg|=h(n,k)^2$ holds only in case (b).

It is natural to consider the analogous product version of the more general Hilton-Milner-Frankl Theorem. To state this result let us define two types of families
\begin{align*}
&\hh(n,k,t) =\left\{H\in \binom{[n]}{k}\colon [t]\subset H, H\cap [t+1,k+1]\neq \emptyset\right\}\cup \left\{[k+1]\setminus \{j\}\colon 1\leq j\leq t\right\},\\[5pt]
&\ha(n,k,t) =\left\{A\in \binom{[n]}{k}\colon |A\cap [t+2]|\geq t+1\right\}.
\end{align*}

\begin{thm}
Suppose that $\hf\subset \binom{[n]}{k}$ is non-trivial and $t$-intersecting, $n>(k-t+1)(t+1)$. Then
\begin{align}\label{thm-5.1}
|\hf| \leq \max\left\{|\hh(n,k,t)|,|\ha(n,k,t)|\right\}.
\end{align}
\end{thm}
For $t=1$, $|\hh(n,k,t)|\geq |\ha(n,k,t)|$ and \eqref{thm-5.1} reduces to the Hilton-Milner Theorem. For $t=2$ and $n>n_0(k,t)$ it was proved in \cite{F782}. For $t\geq 20$ and all $n>(k-t+1)(t+1)$ it follows from \cite{FFuredi2}, however it was first proved in full generality by  \cite{AK}.

Let us conclude this paper by announcing the corresponding product version.

\begin{thm}[\cite{FW2022-2}]
Suppose that $\hf,\hg\subset \binom{[n]}{k}$ are non-trivial and cross $t$-intersecting, $n\geq4(t+2)^2k^2$, $k\geq 5$. Then
\begin{align}\label{thm-5.2}
|\hf||\hg| \leq \max\left\{|\hh(n,k,t)|^2,|\ha(n,k,t)|^2\right\}.
\end{align}
\end{thm}

\end{document}